\documentclass[10pt]{amsart}

\usepackage{hyperref}

\usepackage{amsmath}
\usepackage{amssymb}
\usepackage[shortalphabetic,backrefs]{amsrefs}
\usepackage{enumerate}

\usepackage{xcolor}

\definecolor{cerulean}{rgb}{0,.48,.65} 
\definecolor{magenta}{rgb}{.5,0,.5} 
\definecolor{dred}{rgb}{.5,0,0} 
\definecolor{green}{rgb}{0,.5,0} 
\definecolor{blue}{rgb}{0,0,0.5} 
\definecolor{black}{rgb}{0,0,0} 
\definecolor{dgreen}{rgb}{0,.3,0} 
\definecolor{vdred}{rgb}{.3,0,0} 
\definecolor{red}{rgb}{1,0,0} 
\definecolor{cornellred}{rgb}{0.7, 0.11, 0.11} 
\definecolor{salmon}{rgb}{0.98,0.50,0.45} 
\definecolor{gray}{rgb}{.5,.5,.5} 
\definecolor{seagreen}{rgb}{0.13,0.70,0.67} 
\definecolor{chartreuse}{rgb}{0.40,0.80,0.00}
\definecolor{cornflower}{rgb}{0.39,0.58,0.93} 
\definecolor{gold}{rgb}{0.80,0.68,0.00} 

\definecolor{csugreen}{rgb}{0.12,0.31,0.17}
\definecolor{csugold}{rgb}{0.78,0.76,0.45}

\hypersetup{
   backref=true,
   colorlinks=true,
   citecolor={csugreen},
   linkcolor={cornellred},
}

\numberwithin{equation}{section}
\newtheorem{thm}[equation]{Theorem}
\newtheorem*{thm*}{Theorem}

\newtheorem{prop}[equation]{Proposition}

\newtheorem{lemma}[equation]{Lemma}

\newtheorem{coro}[equation]{Corollary}

\theoremstyle{remark}
\newtheorem*{remark*}{Remark}
\newtheorem{remark}[equation]{Remark}

\theoremstyle{definition}

\def\punct{}
\newtheoremstyle{dotless}{3pt}{3pt}{\itshape}{}{\bfseries}{\punct}{.5em}{}
\theoremstyle{dotless}

\numberwithin{figure}{section}
\numberwithin{table}{section}

\let\hom\relax
\DeclareMathOperator{\hom}{Hom}
\DeclareMathOperator{\Hom}{Hom }
\DeclareMathOperator{\End}{End}
\DeclareMathOperator{\Alt}{Alt }

\DeclareMathOperator{\PSL}{PSL}

\DeclareMathOperator{\GL}{GL}

\DeclareMathOperator{\PSp}{PSp}
\DeclareMathOperator{\PSO}{PSO}
\DeclareMathOperator{\PSU}{PSU}

\DeclareMathOperator{\Gal}{Gal }
\DeclareMathOperator{\Isom}{Isom}

\DeclareMathOperator{\Aut}{Aut}

\DeclareMathOperator{\Prob}{Pr}

\title{Subgroups of simple groups are as diverse as possible}
\author{Martin Kassabov}
\address{
	Department of Mathematics\\
	Cornell University\\
	Ithaca, NY 14850\\
}
\email{martin.kassabov@cornell.edu}

\author{Brady A. Tyburski}
\address{
   Program in Mathematics Education\\
   Michigan State University\\
   East Lansing, MI 48824
}
\email{tybursk2@msu.edu}

\author{James B. Wilson}
\address{
	Department of Mathematics\\
	Colorado State University\\
	Fort Collins, CO 80523\\
}
\email{James.Wilson@ColoState.Edu}
\date{\today}
\thanks{This work was partially supported by NSF grants DMS-1620454 and DMS-1601406,
and the Simons Foundation, with thanks to the Hausdorff Institute for Mathematics
and Isaac Newton Institute (EPSRC Grant Number EP/R014604/1).}

\begin{document}

\begin{abstract}
   For a finite group $G$, let $\sigma(G)$ be the number of subgroups of $G$ and
   $\sigma_\iota(G)$ the number of isomorphism types of subgroups of $G$.

   Let $L=L_r(p^e)$ denote a simple group of Lie type, rank $r$, over a field of
   order $p^e$ and characteristic $p$.  If $r\neq 1$,
   $L\not\cong {}^2 B_2(2^{1+2m})$,
   there are constants $c,d$, dependent on the Lie
   type, such that as $re$ grows
   \begin{align*}
      p^{(c-o(1))r^4e^2}
      & \leq \sigma_{\iota}(L_r(p^e))
         \leq \sigma(L_r(p^e)) \leq p^{(d+o(1))r^4e^2}.
   \end{align*}
   For type $A$, $c=d=1/64$.  For other classical groups
   $1/64\leq c\leq d\leq 1/4$.  For exceptional and twisted groups
   $1/2^{100}\leq c\leq d\leq 1/4$. Furthermore,
   \begin{align*}
      2^{(1/36-o(1))k^2)}
      & \leq \sigma_{\iota}(\Alt_k)
         \leq \sigma(\Alt_k)\leq 24^{(1/6+o(1))k^2}.
   \end{align*}
   For abelian and sporadic simple groups $G$, $\sigma_{\iota}(G),\sigma(G)\in
   O(1)$. In general these bounds are best possible amongst groups of the same
   orders.  Thus with the exception of finite simple groups with bounded ranks
   and field degrees, the subgroups of finite simple groups are as
   diverse as possible.

\end{abstract}

\maketitle

\section{Introduction}
\label{chap:intro}

Recent and historic attention has considered the subgroups of finite simple
groups. In small cases these subgroups can be classified, at least up to
conjugacy, and for general finite simple groups the maximal subgroups have also
been detailed; cf.
\citelist{\cite{Dynkin:group}\cite{Aschbacher}\cite{KL}\cite{BHRD}}. Other works
consider intersections of maximal subgroups
\citelist{\cite{GG}\cite{BW:isom}\cite{BF}} which gives the potential to explore
all subgroups of finite simple groups.
Here we prove bounds on the total number
of distinct isomorphism types of subgroups within finite simple groups.  The
bulk of the variety is witnessed already by the nilpotent subgroups of finite
simple groups.

{\bf Notation.} Throughout this work $p$ is a prime and for a positive integer
$n$, $\nu_p(n)$ is largest $\nu$ such that $p^{\nu}| n$ and
$\mu(n)=\max\{\nu_p(n) : p|n\}$. For a finite group $G$ define $\sigma(G)$ as
the number of subgroups of $G$ and $\sigma_\iota(G)$ as the number of
isomorphism types of subgroups of $G$.  We note that
\begin{align*}
   \sigma_\iota(G)
      \leq \sigma(G)\leq n^{\mu(n)+1}\leq 2^{(\log n)^2}.
\end{align*}
The upper bounds follow from the cumulative work of Gasch\"utz, Kov\'acs,
Guralnick and Lucchini that proves a (sub)group of order $k|n$ is generated by a
set of size $\mu(k)+1\leq \mu(n)+1$ \cite{BNV:enum}*{Theorem~16.6}.

\begin{thm}\label{thm:simple} Let $L=L_r(p^e)$ denote a simple group of Lie type,
   rank $r$, over a field of order $p^e$ and characteristic $p$.
   If $L\not\cong \PSL_2(p^e)$ and $L\not\cong {^2 B_2}(2^{1+2m})$
   then there are constants
   $c,d$, dependent on the Lie type, such that
   \begin{align*}
      c \leq \liminf_{re\to \infty}
                  \frac{\log_p \sigma_{\iota}(L_r(p^e))}{r^4e^2}
         \leq \limsup_{re\to \infty}
                  \frac{\log_p \sigma(L_r(p^e))}{r^4e^2}
         \leq d.
   \end{align*}
   For type $A$, $c=d=1/64$.  For other classical groups
   $1/64\leq c\leq d\leq 1/4$.  For exceptional and twisted groups
   $1/2^{100}\leq c\leq d\leq 1/4$. Furthermore,
   \begin{align*}
      \frac{1}{36}
         \leq
            \liminf_{k\to \infty}
               \frac{\log_2 \sigma_{\iota}(\Alt_k)}{k^2}
         \leq
            \limsup_{k\to\infty}
               \frac{\log_2 \sigma(\Alt_k)}{k^2}
         \leq \frac{\log_2 24}{6}.
   \end{align*}
   For abelian and sporadic simple groups $G$, $\sigma_{\iota}(G),\sigma(G)\in
   O(1)$.
\end{thm}

An alternative formulation of the bounds is offered in the abstract and neither
estimate is as precise as the bounds proved herein.  Also within our
characterization is the requirement that at least one of $r$ or $e$ grow, which
further excludes groups like $\PSL_3(p)$ where only the prime $p$ will grow.  In
general we have not considered the effects of varying primes, and for small rank
and degrees our estimates are likely far from tight, especially for exceptional
and twisted groups.  Remarks~\ref{rem:small-rank-1} \& \ref{rem:best-bounds-U}
speak to the limits of our counting method. We note that
$\sigma_{\iota}(\PSL_2(p^e))\in p^{\Theta(e)}$ whereas $\sigma(\PSL_2(p^e))\in
p^{\Theta(e^2)}$. So the omission of the groups $\PSL_2(p^e)$ form
Theorem~\ref{thm:simple} is required.

On the other hand we expect that the Suzuki groups ${}^2 B_2(2^{1+2m})$ can be
included in the statement of Theorem~\ref{thm:simple}. Indeed, we can prove that
these groups have many subgroups provided that there exists a small dimensional
subspace in the field $\mathbb{F}_{2^{1+2m}}$ which ``generates'' the field
under some operation---unlike other simple groups, in the Suzuki case the
commutator map in the Sylow $2$-group is not related to the field multiplication
which prevents us from using standard techniques to prove the existence of such
subspaces for general $e=1+2m$. However, we were able to exhibit such  a
subspace for $e<1000$.  Hence, we believe that the class of Suzuki groups will
eventually be included in a statement like that of Theorem~\ref{thm:simple}.

As a consequence of our proof of Theorem~\ref{thm:simple} we also bound general
classes $\mathfrak{X}_n$, resp. $\mathfrak{X}_{\leq n}$, of groups of order $n$,
resp. at most order $n$.  We extend $\sigma$ to
$\sigma(\mathfrak{X}_n)=\max\{\sigma(G)\mid G\in \mathfrak{X}_n\}$ and
$\sigma(\mathfrak{X}_{\leq n})$.  When $\mathfrak{X}_n$ is all groups of order
at most $n$ we write simply $\sigma(n)$. Do likewise with $\sigma_\iota$.
\begin{thm}\label{thm:sigma-iota-general}
   For the class $\mathfrak{N}_n$ of nilpotent groups of order $n$,
   \begin{align*}
      \sigma_{\iota}(\mathfrak{N}_n),\sigma(\mathfrak{N}_n)\in \prod_{p|n}p^{\nu_p(n)^2/4+\Theta(\nu_p(n))}.
   \end{align*}
   Indeed, there are at least $\prod_{p|n} p^{\nu_p(n)^3/108+\Theta(\nu_p(n)^2)}$ pairwise
   non-isomorphic groups $G$ of order $n$ having
   $\sigma_{\iota}(G)\in \prod_{p|n} p^{\nu_p(n)^2/36+\Omega(\nu_p(n))}$.
   As a consequence for the class of nilpotent groups of order at most $n$ we have
   $$
      \sigma_{\iota}(\mathfrak{N}_{\leq n}),\sigma(\mathfrak{N}_{\leq n})\in 2^{(\log_2 n)^2/4  +\Theta(\log_2(n))}.
   $$
\end{thm}

For convenience we confine all our calculations to the context of finite simple
groups but our method extends to almost simple, quasi-simple, algebraic, and
Steinberg groups.  In the latter cases one may replace cardinality with
dimensions of varieties.  If one assume the Classification of Finite Simple
Groups (CFSG) (\cite{CFSG}) then Theorem~\ref{thm:simple} indeed concerns all
simple groups, but we have no explicit dependence on that theorem within our
proof.

\subsection*{Acknowledgements}
We are grateful to Bob Guralnick and Bill Kantor for answers to questions on
simple groups.  Part of this work comprised the Master's
Thesis at Colorado State University of the second author \cite{Tyburski}.

\subsection{Proof overview.}

Throughout $K=\mathbb{F}_q$ is a finite field of characteristic $p$ and order
$q=p^e$, and $k=\mathbb{F}_p$.  
Abbreviate $\otimes:=\otimes_{\mathbb{Z}}$ and $\hom(-,-):=\hom_{\mathbb{Z}}(-,-)$.
We identify biadditive maps $U\times V\to W$ with linear maps $U\otimes V\to W$.
We let $\End(V)$ denote endomorphisms and $\Aut(V)$ automorphisms.

To count subgroups of a finite group $G$ the usual strategy is linear algebra.
Select a large $p$-elementary abelian section $L_1=\Gamma_1/\Gamma_2$ of $G$
were the counting reduces to counting subspaces of a vector space.  If $\log_p
|L_1|\in \Theta(\log_p |G|)$ then the number of subspaces matches the targeted
quantity of $p^{\Theta(\nu(|G|)^2)}$.  Such a count gives no clue about the
possible range of isomorphism classes of subgroups and indeed if the only large
$p$-elementary abelian sections have $\Gamma_2=1$ then this process surveys just
$O(\log_p |G|)$ isomorphism types.

To obtain a larger number of isomorphism classes we appeal to multi-linear
algebra.  First we select a series $\Gamma_1>\Gamma_2>\Gamma_3>\cdots$ arranged
into a filter in the sense that for all $i,j$, $[\Gamma_i,\Gamma_j]\leq
\Gamma_{i+j}\leq \Gamma_i\cap \Gamma_j$. This allows for the creation of an
associated graded Lie algebra $L=\bigoplus_i \Gamma_i/\Gamma_{i+1}$, with brackets
$[,]_{i,j}: L_i \otimes L_j \to L_{i+j}$ induced from commutation in the group $\Gamma_1$.

We focus specifically in filters where $d_i:=\log_p |L_i|$ satisfies $d_1^2/4\gg
d_2^2+d_3^2$, which guarantees a larger number of subgroups $Q\leq \Gamma_1$
containing $\Gamma_2$. We can restrict $[,]_{1,2}$ to $Q$ and obtain a map
$Q/\Gamma_2\otimes L_2\to L_3$. Said another way, one can now consider subspaces
$Q/\Gamma_2 \hookrightarrow L_1 \to \hom(L_2,L_3)$ under the action of
automorphisms $\Aut(L_2)\times \Aut(L_3)$.  Now the intuition of the count
becomes clear. There are $p^{\ell (d_1-\ell)}$ subspaces of dimension
$\ell=\mathrm{rank}~Q/\Gamma_2$. Meanwhile $\Aut(L_2)\times \Aut(L_3)$ has
only $d_2^2+d_3^2$ parameters.  For $\ell\approx d_1/2$, $\ell(d_1-\ell)\gg d_2^2+d_3^2$;
thus, there are many orbits of subspaces in $L_1$ under the action of the
automorphism group. Such a count is a variation on the method introduced by
Higman to estimate the number of isomorphism types of finite
$p$-groups~\cite{BNV:enum}*{Chapter~4}.

The subtlety hidden here is that upon restricting to subgroups $Q\leq \Gamma_1$
containing $\Gamma_2$ there is no immediate requirement that such $Q$ will
determine $\Gamma_i$, for $i>1$ calling into question why isomorphisms between
such $Q$ should restrict to isomorphisms of the biadditive maps
$Q/\Gamma_2\otimes L_2\to L_3$.  Indeed this is not true in general.

We will focus on the cases where most subgroups $Q$ determine the subgroups
$\Gamma_i$.  For general bounds we look for examples where $[,]_{1,2}:R\otimes
M\to M$, where $R$ is an algebra and $M$ is a left $R$-module.
We consider subgroups $Q$ together with some additional data,
that can be used to reconstruct the subgroups
$\Gamma_i$. This extra data leads to significant under count but is
straight-forward. To obtain counts that apply in small ranks we instead recover
appropriate rings that act on $[,]_{1,2}$ but are not part of the commutation
themselves.

Upper bounds are established by a result of Wall~\cite{Wall} that shows that the
number of $p^k$-order subgroups of a group of order $p^n$ is at most the
number of $k$-dimensional subspaces of $\mathbb{F}_p$-vector spaces of dimension
$n$. This combined with structure of maximal solvable groups of matrix groups
affords a tight bound in type $A$ and a suitable bound for other groups.

\section{Module Nurseries \& Kinder}\label{sec:nurseries}

A filtration $\Gamma_1>\Gamma_2>\Gamma_3>\cdots$ of subgroups satisfies
$[\Gamma_i,\Gamma_j]\leq \Gamma_{i+j}\leq \Gamma_i\cap\Gamma_j$.  Thus,
$L_i=\Gamma_i/\Gamma_{i+1}$ forms a graded Lie ring $L=\bigoplus_i L_i$. Call
this filtration a \emph{nursery} if $L_1$ has more subspaces than the size of
$\Aut(L_2)\times \Aut(L_3)$.  By a \emph{module} nursery we mean there is an
(unital) associative ring $R$, a faithful (left) $R$-module $M$, and
isomorphisms $\alpha: L_1 \to R$; $\beta: L_2 \to M$ and $\gamma: L_3 \to M$
where
\begin{align}\label{eq:alpha-beta-gamma}
    \gamma([u\Gamma_2,v\Gamma_3]\Gamma_4) & = \alpha(u\Gamma_2)\beta(v\Gamma_3).
\end{align}
We call the nursery \emph{exact} if $\Gamma_4 = [\Gamma_2,\Gamma_2]$, this
condition is automatically satisfied if $\Gamma_4 = 1$.  A primary example
of nurseries are the generalized Heisenberg groups $\mathcal{H}_{abc}(K)$
defined as the block-upper triangular matrices inside $\GL_d(K)$, $d=a+b+c$:

\begin{align}\label{def:hei}
    \mathcal{H}=\mathcal{H}_{abc}(K) & =
    \left\{
        \begin{bmatrix}
            I_a & U & W \\ 0 & I_b & V \\ 0 & 0 & I_c
        \end{bmatrix}
        ~\middle|~
        \begin{array}{rcl}
            U & \in & \mathbb{M}_{a\times b}(K)\\
            V & \in & \mathbb{M}_{b\times c}(K)\\
            W & \in & \mathbb{M}_{a\times c}(K)
        \end{array}
    \right\}
\end{align}

Suppose $a\geq c$.  Then these groups have a filtration defined by
\begin{align*}
    \Gamma_1=\mathcal{H} >
    \Gamma_2=
    \left\{
        \begin{bmatrix}
            I_a & 0 & W \\ 0 & I_b & V \\ 0 & 0 & I_c
        \end{bmatrix}
    \right\}
    >
    \Gamma_3=
    \left\{
        \begin{bmatrix}
            I_a & 0 & W \\ 0 & I_b & 0 \\ 0 & 0 & I_c
        \end{bmatrix}
    \right\}
    > \Gamma_{c\geq 4} = 1.
\end{align*}
If $a=b$ then $L_1\cong \mathbb{M}_a(K)$, $L_2\cong L_3\cong \mathbb{M}_{a\times
c}(K)$ and $[,]_{1,2}:L_1\otimes L_2\to L_3$ is equivalent to the module action
of $R=\mathbb{M}_a(K)$ on the left of $\mathbb{M}_{a\times c}(K)$.  If $a>2c$
then the number of subspaces of $L_1$ is at least $p^{a^4/4}$ while
$\Aut(L_2)\times \Aut(L_3)$ has order at most  $2p^{a^2 c^2}$; so, these are
exact module nurseries.

Already observed in \cite{Wilson:profiles}*{Section~3}, for each $\nu\geq 3$,
the groups $\mathcal{H}_{abc}(k)$ with $a=b=\lceil\nu/3\rceil$, $c=1$ contain
$p^{\nu^3/27+\Theta(\nu^2)}$ pairwise non-isomorphic subgroups of order $p^{\nu}$. So
these groups have a diverse family of subgroups but we shall need many more
subgroups to obtain meaningful lower bounds. The main result in this section is
to generalize such counts by replacing generalized Heisenberg groups with the
concept of exact module nurseries. 

Throughout this and the next section we concentrate on subgroups $Q$ where
$\Gamma_2\leq Q\leq \Gamma_1$ for a nursery $\Gamma_*$.  We call these
subgroups \emph{kinder} (or \emph{kind} for one).

\begin{thm}\label{thm:nursery-count} Let $\Gamma_1>\Gamma_2>\Gamma_3>\cdots$ be
    an exact module nursery for a $k$-algebra $R$ and $R$-module
    $M$. Fix a
    generating set ${\sf S}$ containing $1$ of $R$ as a $k$-algebra, and a set
    ${\sf T}\subset M$ such that $\bigcap_{x\in {\sf T}}\mathrm{Ann}_R(x)=0$.

    Then the number of isomorphism types of kinder $Q$ such that
    ${\sf S} \subset Q/\Gamma_2$ and $|Q/\Gamma_2| = p^\ell$ is at least
    \begin{align*}
        p^{(\ell-s) (r - \ell) - \ell s - mt}
    \end{align*}
    where $|R|=p^r$, $s=|{\sf S}|$, $|M|=p^m$, and $t=|{\sf T}|$.
\end{thm}

\begin{proof}
    Fix a kind $Q$. Recall the meaning of $(\alpha,\beta,\gamma)$ from
    \eqref{eq:alpha-beta-gamma}. Fix transversals $V_i\leq \Gamma_i$ for
    $\Gamma_i/\Gamma_{i+1}$. We consider the tuples $(Q,\rho,\mu)$, where
    $Q\leq \Gamma_1$ containing $\Gamma_2$, $\rho:{\sf S}\to V_1\cap Q$ ($\rho$
    for ring) where $\alpha(\rho(s)\Gamma_2)=s$ for all $s\in {\sf S}$ and
    $\mu:{\sf T}\to V_2$ ($\mu$ for module) such that $\beta(\mu(t)\Gamma_3))=t$
    for all $t \in {\sf T}$. We claim that the data $(Q,\rho,\mu)$ is enough to
    reconstruct each $\Gamma_i$.\footnote{Equivalently consider mappings
    $[\rho]:{\sf S}\to Q/\Gamma_2$ and $[\mu]:{\sf T}\to Q/\Gamma_3$ and show
    the construction is unaffected by choice of coset representatives.}

    Set $X = \bigcap_{t\in {\sf T}}\{q\in Q\mid [q,\mu(t)]\leq \Gamma_4\}$. As
    $[\Gamma_2,\Gamma_2]\leq \Gamma_{4}$, and because $\mu(t)\in \Gamma_2$, it
    follows that $\Gamma_2\leq X$. Suppose that $q\in Y$, if $q \not\in
    \Gamma_2$ then $\alpha(q\Gamma_2) \neq 0$ and there exists $t \in {\sf T}$
    such that $\alpha(q\Gamma_2)  t \neq 0$, which is equivalent to
    $\gamma([q, \mu(t)]\Gamma_4 ) \neq 0$, but this contradicts the
    assumption that $q\in X$. This shows that $X=\Gamma_2$.

    Next, put $Y=[\rho(1),X]$ and $Z = [X,X]$. It follows from the definition of
    exact module nursery that $Y = \Gamma_3$ and $Z = \Gamma_4$.

    It remains to reconstruct $\Gamma_1$, but $\Gamma_1$ is a super group so
    here we mean simply that data $(Q,\rho,\mu)$ identifies how $Q/\Gamma_2$
    sits in $L_1$. Therefore the preimage of that embedding is fixed by the data
    provided. For that observe the usual additive mapping
    $Q/\Gamma_2\hookrightarrow \hom(L_2,L_3)$ is an embedding because the kernel
    $X=\Gamma_2$.  So we obtain an embedding $\chi:Q/\Gamma_2\to \End(M)$,
    relative to $(\beta,\gamma)$. Lastly, for $s\in {\sf S}$ and $m\in M$,
    $\chi(\rho(s))(\beta^{-1}(m))=\gamma[\rho(s),\beta^{-1}(m)]=sm$. Since
    $R=k\langle {\sf S}\rangle$, the $k$-algebra generated by the image of
    $\chi$ contains the image of $R$ in $\End(M)$, and because the image of
    $\chi$ consists of the $\gamma$-image of commutators, $\chi(Q)$ is contained
    in the image of $R$ in $\End(M)$. Thus, $\chi$ leads to a unique embedding
    of $Q/\Gamma_2$ into $R\cong \Gamma_1/\Gamma_2$. Therefore $(Q,\rho,\mu)$
    determines the embedding of $Q$ into $\Gamma_1$.

    In order to estimate the number of kinder, count the number
    of subspaces of $R$ which contain ${\sf S}$, which is at least
    $p^{(\ell-s) (r - \ell)}$ which is lower bound for the total number of
    $(Q,\rho,\mu)$ triples.

    The choice for $\rho:{\sf S} \to V_1\cap Q$ is (at most) $p^{\ell s}$ and
    the choices for $\mu:T\to V_2$ are $p^{mt}$, therefore at most $p^{\ell s +
    mt}$ triples correspond to isomorphic groups $Q$. Thus the number of
    isomorphism types of kinder is at least $p^{(\ell-s)(r-\ell)-\ell s-mt}$.
\end{proof}

\begin{coro}\label{coro:U_d}
    Fix $p$, for $d\geq 5$, the following holds:
    \begin{align*}
        p^{(1/64-o(1))d^4e^2} \leq
        \sigma_{\iota}(U_d(\mathbb{F}_{p^e}))
        \leq \sigma(U_d(\mathbb{F}_{p^e}))
        \leq p^{(1/64+o(1))d^4e^2}.
    \end{align*}
    For $d\in \{3,4\}$ the following holds:
    \begin{align*}
        p^{(1/4-o(1))e^2} \leq \sigma_{\iota}(U_d(\mathbb{F}_{p^e})).
    \end{align*}
    Meanwhile $\sigma(U_3(\mathbb{F}_p^e))\leq p^{3e^2/4+O(e)}$ and
    $\sigma(U_4(\mathbb{F}_{p^e}))\leq p^{3e^2/2+O(e)}$.
    For $d=2$, $\sigma_{\iota}(U_d(\mathbb{F}_{p^e}))=e+1$ and
    $\sigma(U_d(\mathbb{F}_{p^e}))\in p^{e^2/4+O(e)}$.
\end{coro}
\begin{proof}
    The group $U_d(K)$ contains $\mathcal{H}_{abc}(K)$ where $a=b$ and $c=1$
    which has an exact module nursery with $R = \mathbb{M}_a(K)$ and $M = K^a$.
    Fix a generator $\omega$ for $K/k$ so that $R=k\langle S\rangle$ where ${\sf
    S}=\{I_a, \omega_1 E_{1,1}, \sum_i E_{i,(i \mod a)+1}\}$.  For ${\sf T}$
    choose a basis of $M$ over $K$. Subject to the constraint $d=2a+1$, by
    Theorem~\ref{thm:nursery-count},
    \begin{align*}
        \sigma_{\iota}(U_d(K))\geq \sigma_{\iota}(\mathcal{H}_{aa1}(K))
        & \geq \max_{\ell} p^{(\ell-3)(a^2 e - \ell) - 3\ell - a^2 e}.
    \end{align*}
    The maximum is achieved when $\ell = \lfloor a^2 e/2 \rfloor$.  Since
    $a=\lfloor (d-1)/2\rfloor$ this yields a lower bound on
    $\sigma_{\iota}(U_d(\mathbb{F}_{p^e}))$
    of $p^{1/64 (d-1)^4 e^2 - (d-1)^2 e}$.  As $|U_d(\mathbb{F}_{p^e})|=p^{\nu}$
    where $\nu=\binom{d-1}{2}e=d^2e/4 +O(de)$. By Wall's theorem \cite{Wall}
    $\sigma(U_d(\mathbb{F}_{p^e}))\leq \sigma(\mathbb{F}_p^{\nu})\leq p^{\nu^2/4+O(\nu)}$.
    So $\sigma(U_d(\mathbb{F}_{p^e}))\leq p^{(1/64+o(1))d^4 e^2}$.
\end{proof}

\begin{remark}\label{rem:small-rank-1}
    The bounds hidden in $o(1)$ can be resolved into the following (some which will be
    improved by our next estimate). With $p$ fixed,
    $\sigma_{\iota}(U_d(\mathbb{F}_{p^e}))\geq p$ in the following cases:
    $d \in \{3,4\}$ and $13 \leq  e$; $d\in \{5, 6\}$ and $5 \leq e$;
    $d \in \{7,8,9,10\}$ and $2\leq e$; and, $d \geq 11$ and $1  \leq e$.
\end{remark}

\section{General nurseries}

Having described the general bound we look here to improve the lower bounds by
proving a stronger property about most subgroups of generalized Heisenberg
groups $\mathcal{H}:=\mathcal{H}_{abc}(K)$ and the filters $\Gamma_*$ we
introduced in Section~\ref{sec:nurseries}.  Note that in this section
$(a,b,c)$ can be arbitrary and in general $\Gamma$ is an exact nursery but not
typically a module nursery. Indeed the associated Lie algebra $L=\bigoplus_i
\Gamma_i/\Gamma_{i+1}$ recovers general matrix multiplication as the bracket
$[,]_{1,2}:L_1\otimes L_2\to L_3$ instead:
\begin{align*}
  \mathbb{M}_{a\times b}(K) \otimes \mathbb{M}_{b\times c}(K)
    \to \mathbb{M}_{a\times c}(K).
\end{align*}

First a few remarks on the canonicity of the choice of $\Gamma_i$: the subgroup
$\Gamma_3$ is the commutator subgroup of $\mathcal{H}$, and therefore is 
characteristic; identifying $E:=\Gamma_2$ is more subtle.  There is a competing choice of
subgroup $\Gamma_2 < F < \Gamma_1$ where (in the notation of \eqref{def:hei})
$V=0$ and $U\in \mathbb{M}_{a\times b}(K)$.  In fact these coordinates are not
in general group theoretic features so there could be many further choices.

The property we need is on pairs of subgroups.  Witness that
$E$ and $F$ are abelian subgroups such that $\mathcal{H}=\langle E,F\rangle$,
$E\cap F =[\mathcal{H},\mathcal{H}]\leq Z(\mathcal{H})$. Such so-called
\emph{hyperbolic pairs} were first studied by Brahana \cite{Brahana}.  
In~\cite{BMW:genus2}*{Lemma~3.5} a characterization of hyperbolic pairs showed they
are in bijection with specific idempotents of a ring $\mathcal{M}$ 
(see~\eqref{def:adj} below) that combined with~\cite{Wilson:Skolem-Noether}*{Corollary~1.5} 
implies that hyperbolic pairs in
$\mathcal{H}_{abc}(K)$ are in the same orbit under the action of the
automorphism group. Thus, the assumption that $a\geq c$ stipulates that we take
the smaller of the two terms in any hyperbolic pair, or if a $a=c$ to pick any
of the terms.
In fact, when $a,c>1$ there is exactly one hyperbolic pair for
$\mathcal{H}_{abc}(K)$, and if furthermore $a>c$ then $\Gamma_2$ is the unique
smallest subgroup in this hyperbolic pair.  When $a=c=1$ the automorphism
group is unusually large and thus there are many hyperbolic pairs.

The main result in this section shows that for a generic kinder $Q$ of
$\Gamma_*$ the same argument applies and the subgroups $\Gamma_3$ are
characteristic in $Q$ and $\Gamma_2$'s are in a single $\Aut(Q)$-orbit
(Proposition~\ref{prop:generic-Q}). In addition all isomorphisms between such
subgroups come from automorphisms of $\mathcal{H}$ fixing $\Gamma_2$. It is
interesting to note that this result also holds in when $c=1$ and $e>7$ even
though in this case $\Gamma_2$ is far from being a characteristic subgroup of
$\mathcal{H}$.

We prove that under mild conditions for $a,b,c,e$, isomorphisms between generic
kinder extend to automorphisms of $\Gamma_1$.  Our version of generic is
measured as a probability but can also be generic in the sense of algebraic
geometry.

\begin{thm}\label{thm:generic} Fix the nursery $\Gamma_*$ of
    $\mathcal{H}_{abc}(\mathbb{F}_q)$ of Section~\ref{sec:nurseries}. If $a\leq
    b$ and $\ell>2+b/a$, or $b<a$ and $\ell>2+a/b$, then amongst kinder $Q$ of
    $\Gamma_*$ with $[Q:\Gamma_2]=p^{\ell}$,
    \begin{align*}
        \Prob(\exists \alpha\in \Aut(\Gamma_1), \alpha(Q)=\tilde{Q}\mid Q\cong \tilde{Q})
        \geq 1-O(1/p).
    \end{align*}
\end{thm}

That result leads to our counting claims, first one about $\mathcal{H}_{abc}(K)$.

\begin{coro}\label{coro:U-count}
    For some constants $C,D>0$
    $$
    p^{(abe)^2/4-C abe}\leq \sigma_{\iota}(\mathcal{H}_{abc}(\mathbb{F}_{p^e}))
    \leq p^{(ab+bc+ac)^2e^2/4+D(ab+bc+ac)e}.
    $$
    Thus, for $a,b\in d/2+O(1)$, and $c\in O(1)$,
    we find $\sigma_{\iota}(\mathcal{H}_{abc}(\mathbb{F}_{p^e}))\in
    p^{d^4e^2/64+\Theta(d^2e^2)}$.
\end{coro}
\begin{proof}
    Without loss of generality let $a\geq c$.
    For the lower bound observe that $\Gamma_1/\Gamma_2\cong \mathbb{M}_{a\times
    b}(K)$ as additive groups and the subgroups $Q$ where $\Gamma_2<Q<\Gamma_1$
    are enumerated by $\mathbb{F}_p$-subspaces $V$ of $\mathbb{M}_{a\times
    b}(K)$, where $|K|=p^e$ with $|C|=p^{\ell}$.  That yields
    $p^{\ell(abe-\ell)}$ choices of $Q$. Assuming $Q$ is generic, the set of
    generic subgroups isomorphic to $Q$ are in the same orbit of the subgroup
    $A:=\Aut(\mathcal{H}_{abc}(K))$ which fix $\Gamma_2$. From the structure of
    $A$ given in \cite{Wilson:Skolem-Noether}*{Corollary~1.5}, it follows that
    this action factors through\footnote{In the special case $a>c=1$ the full
    automorphism group does not preserve $\Gamma_2$ but the stabilizer of
    $\Gamma_2$ is nevertheless the group described.  We note that that work fails
    to report the obvious graph automorphism in the case $a=c$.  This does not
    affect our count but we include the correction for completeness.}
    \begin{align*}
        \Gal(K)\ltimes (\GL_a(K)\times  \GL_b(K)\times \GL_c(K))/K^{\times},
            &\textnormal{ if } a>c\geq 1;\\
        2.\Gal(K))\ltimes (\GL_a(K)\times  \GL_b(K)\times \GL_a(K))/K^{\times},
            & \textnormal{ if }a=c>1;\\
        \Gal(K)\ltimes {\rm GSp}_{2b}(K), &\textnormal{ if } a=c=1.
    \end{align*}
    Hence, each orbit has cardinality at most
    $p^{O((a^2+b^2+c^2)e)}$.
    Maximizing over $\ell= (1+o(1))\frac{abe}{2}$,
    there are at least $p^{(abe)^2/4+O(abe)}$ distinct orbits.

    The upper bound comes from Wall's theorem \cite{Wall}.
\end{proof}

\begin{remark}\label{rem:best-bounds-U}
    Improving the bounds from Remark~\ref{rem:small-rank-1} from
    Corollary~\ref{coro:U-count} one obtains that $p\leq
    \sigma_{\iota}(U_d(p^e))\in p^{(1/64-o(1))d^4e^2}$ whenever $d=3$ and $e\geq
    9$, $d=4$ and $e\geq 5$, $d=5$ and $e\geq 3$, $d=6$ and $e\geq 2$, and in
    general when $d\geq 7$. Furthermore the secondary error term of
    Corollary~\ref{coro:U_d} is improved from $-Cd^3e^2$ to $-Dd^2 e$.
\end{remark}

\subsection{Some probabilistic estimates}
\label{sec:prob}

\begin{lemma}
    \label{lm:prob-span}
    Let $v_1,\ldots ,v_s$ be independently random vectors of an
    $n$-dimensional vector space $V$ over a finite field of order $q$.
    $$
    \Prob(V=\langle v_1,\ldots,v_s\rangle) \geq 1 - \frac{q^{n-s}-q^{-s}}{q -1}
        \geq 1 - q^{n-s}.
    $$
\end{lemma}
\begin{proof}
    There are $(q^n-1) / (q-1)$ maximal subspaces, and the probability that a
    random vector is in a given maximal subspace is $q^{-1}$. So the
    probability that all $\{v_i\}$ are in a fixed maximal subspace is $q^{-s}$.
    Multiplying this by the number of maximal subspaces gives an upper bound of
    the probability that the span of $\{v_i\}$ is a proper subspace.
\end{proof}
\begin{remark}
    When the field $K$ is infinite say instead that the variety of $s$
    tuples of vectors in $K^n$ which do not span the whole space has codimension
    $n-s+1$.
\end{remark}

For $(s\times b)$-matrices $\Phi_1,\ldots,\Phi_c$ and $(a\times t)$-matrices
$\Upsilon_1,\ldots,\Upsilon_c$ over $K$, we have a  $K$-vector space
\begin{align}\label{def:adj}
    \hom(\Phi_*,\Upsilon_*) & = \{
        (A,B)\in \mathbb{M}_{a\times s}(K)\times\mathbb{M}_{b\times t}(K)\mid
                    (\forall i)(A\Phi_i=\Upsilon_i B^t)\}.
\end{align}
As the notation suggests these are morphisms in an abelian category (though not
in general a module category \cite{Wilson:division}) and so
$\End(\Phi_*)=\hom(\Phi_*,\Phi_*)$ is a ring.  This is in fact the ring $\mathcal{M}$
alluded to in the introduction of this section.  We now quantify the generic
expectation of these morphisms sets.

\begin{thm}
    \label{thm:prob-end}
    Let $m\le n$ and $\Phi_1,\ldots,\Phi_s$ be independently random
    $(m\times n)$-matrices over a finite field $K$ of order $q$.
    If $2+n/m \leq s$ then
    \begin{align*}
    & & \Prob(\End(\Phi_*)\cong K) &  \geq 1-O(1/q) \\
    & (m<n\textnormal{ or }2<m=n) & \Prob(\hom(\Phi_*,\pm \Phi^t_*)= 0 ) &  \geq 1-O(1/q).\\
    \end{align*}
\end{thm}

\begin{proof}
    The condition that $\dim \hom(\Phi_*,\Gamma_*)>c$ is an algebraic condition
    in the entries of $(\Phi_*,\Gamma_*)$. If this condition defines a proper
    subvariety of codimension $g$, then the number of points is less than $C
    q^{mns - g}$, which implies the resulting bound.

    Since $K \subset \End(\Phi_*)$,  to show that this is a proper subvariety it
    suffices to construct an example of $\Phi_*$ with $\End(\Phi_*)\cong K$,
    i.e. a generic point off the variety. Without loss of generality we can
    assume that $m\leq n$ (otherwise we can take transpose).

    Let $m=n$ then we can view $\mathbb{M}_m(K)$ as generated as a $K$-algebra by
    $1, \alpha,\sigma$ where $E:=K\langle 1,\alpha\rangle$ is a field extension
    of degree $m$ (this step assumes $K$ is a finite field) and $E^{\sigma}=E$
    induces a field automorphism of $E$ that generates the Galois group of
    $E/K$.  For $\Phi_*$ we
    can take $\Phi_1 = I_m$, $\Phi_2=\alpha$ and $\Phi_3=\sigma$. Then
    $\End(\Phi_*)$ is the centralizer of the algebra generated by $\Phi_*$
    therefore we have  $\End(\Phi_*) \cong K$. If $m >2$ then $\sigma(\alpha)
    \not = \sigma^{-1}(\alpha)$, which implies that $\hom(\Phi_*,\pm\Phi_*^t) =
    0$.
        

    Now suppose $m<n$ and define
    \begin{align*}
     \Phi_0 & = \begin{bmatrix}  I_m & 0\end{bmatrix}\in \mathbb{M}_{m\times n}(K), \\
     \Phi_i & = \begin{bmatrix} 0_{m\times (1 + m(i-1))} &  I_m &  0\end{bmatrix}\in \mathbb{M}_{m\times n}(K) \quad \mbox{for } 1\leq i \leq (n-1)/m,\textnormal{ and} \\
     \Phi_{\infty} & = \begin{bmatrix} 0_{m\times (n-m)} &  I_m \end{bmatrix}\in \mathbb{M}_{m\times n}(K).
     \end{align*}
    Notice that $\cap_i \ker \Phi_i =0$ which implies that for a given matrix
    $A$ there is at most one matrix $B$ such that for all $i$,
    $A \Phi_i = \Phi_i B^t$.
    Define the subspaces $U_j, U'_j,W_j \leq K^m$ inductively by
    \begin{align*}
    U_1 &= \Phi_2(\ker \Phi_1)  &       U_{j+1} & = \Phi_2( \Phi_1^{-1}(U_j)) \\
    U'_1 &= \Phi_1(\ker \Phi_2)    &  U'_{j+1} & = \Phi_1( \Phi_2^{-1}(U'_j)) \\
    W_j &= U_j \cap U_{m+1-j}
    \end{align*}
    Then $\dim U_j = \dim U'_j =j$  and $\dim W_j=1 $ for $1\leq j\leq m$ and
    $\Phi_1 \Phi_2^{-1}$ induces an isomorphism $\gamma_j : W_j \to
    W_{j+1}$. The condition $(\forall i)(A \Phi_i = \Phi_i B^t)$ implies
    that $A U_j \leq  U_j$ and $A U'_j \leq U'_j$. Therefore $A$ sends the
    $1$-dimensional spaces $W_j$ to themselves and it is compatible with the
    isomorphism $\gamma_i$. Since $K^m  = \bigoplus W_j$ this implies that $A$
    is a scalar matrix, and so is $B$. This shows that $\End(\Phi_*)\cong K$.

    To see that $(A,B)\in \hom(\Phi_*,\pm \Phi_*^t)$ is $0$,
    $A\Phi_i  = \pm \Phi_i^t B^t$ implies that the first rows $A_i$
    of $A$ are $0$ wherever $\Phi_i^t$ has a $0$ row, and wherever
    $\Phi_i^t\neq 0$, $A_{1+m(i-1)+k}=B_k$, and column $B^j$ of $B$ is $0$
    wherever $\Phi_i$ has 0 columns.  Running over $i,j$, $A=0$ and $B=0$.
\end{proof}

For $m=n$ the bound of $s\geq 3$ is best possible and for $m=1$, $s=n+1$
is best possible.  Similar to results appear as far back as Kronecker, one needs a
more nuanced tool than linear algebra to prove the result generically. See for
example~\cite{GG}.  The result above may be classically known though we did not
find a version to cite.  These enable the following computation related to
the commutator of generic kinder of $\mathcal{H}_{abc}(K)$.

\begin{coro}
    \label{coro:prob-end}
    Fix $a\leq b$ and $c\geq 2+b/a$, or $b\leq a$ and $c\geq 2+a/b$.
    Given random $\Phi_1,\ldots,\Phi_c\in \mathbb{M}_{a\times b}(K)$ define
    \begin{align*}
        & (\forall v) & \Lambda_v & =
        \begin{bmatrix} 0 & \Phi_v\\ -\Phi_v^t & 0 \end{bmatrix}.
    \end{align*}
    Then if $a=b$ with high probability $\End(\Lambda_*)\cong \mathbb{M}_2(K)$
    and if in $a\neq b$ then with high probability $\End(\Lambda_*)\cong
    (K\oplus K)\ltimes J$ where $J$ is the Jacobson radical.
\end{coro}
\begin{proof}
    Without loss of generality let $a\leq b$.
    The equations defining $\End(\Lambda_*)$ say
    \begin{align*}
        \begin{bmatrix}
            A_{11} & A_{12} \\ A_{21} & A_{22}
        \end{bmatrix}
        \begin{bmatrix}
            0 & \Phi_v \\ -\Phi_v^t & 0
        \end{bmatrix}
        =
        \begin{bmatrix}
            0 & \Phi_v \\ -\Phi_v^t & 0
        \end{bmatrix}
        \begin{bmatrix}
            B_{11} & B_{12} \\ B_{21} & B_{22}
        \end{bmatrix}^t.
    \end{align*}
    Expanding these equations we get the $4$ defining properties
    \begin{align*}
        (A_{12}, B_{12})& \in \Hom(-\Phi_*^t,\Phi_*)
        &
        (A_{11}, B_{22})& \in \End(\Phi_*) \\
        (A_{22}, B_{11})& \in \End(-\Phi_*^t,\Phi_*^t)
        &
        (A_{21}, B_{21})& \in \Hom(\Phi_*,-\Phi_*^t).
    \end{align*}
    Theorem~\ref{thm:prob-end} implies that with high probability,
    $\End(\Phi_*)\cong \hom(\Phi_*,-\Phi_*^t)\cong \hom(\Phi_*^t,\Phi_*)\cong K$
    if $a=b$, and if $a\neq b$ then
    $\End(\Phi_*)\cong K\cong\End(\Phi_*^t)$ and $\hom(\Phi_*,-\Phi_*^t)=0$.
    The result follows.
\end{proof}

\subsection{Lifting isomorphisms; Proof of Theorem~\ref{thm:generic}}\label{sec:lift}

Now we setup the mechanics to proof Theorem~\ref{thm:generic}. Our process is in
two steps.  First we show that generically $\Gamma_3$ is a characteristic
subgroup of $Q$. Second we show the $\Gamma_2$ come from a unique orbit under
the automorphism group of $\Aut(Q)$ and that the embedding from $Q/\Gamma_2$
into $\mathbb{M}_{a\times c}(K)$ is unique up to the action by ${\rm\Gamma
L}(K^a)\times {\rm \Gamma L}(K^c)$.

\begin{prop}\label{prop:generic-Q-1} Let $Q$ range over kinder of
    $\mathcal{H}_{abc}(K)$, with respect to the nursery $\Gamma_*$, and such
    that $|Q/\Gamma_2| = p^\ell$ with $\ell>b/a$. Then with high probability
    $[Q,Q] = \Gamma_3$.
\end{prop}
\begin{proof}
    The condition $[Q,\Gamma_2] = \Gamma_3$ is equivalent to saying: the $K$ span
    of the elements in $Q/\Gamma_2$,  viewed as elements in
    $\mathbb{M}_{a\times b}(K)$, generate $K^b$. The columns of the generators
    of $Q/\Gamma_2$ are random vector in $K^b$; thus, it suffices that the total
    number of columns is at least $b$.  This happens when $\ell > b/a$.
    Following Lemma~\ref{lm:prob-span} this happen with probability at least
    $1 - q^{b - a\ell}$.

    As $ [Q,\Gamma_2] \leq [Q,Q] \leq [\Gamma_1,\Gamma_1] = \Gamma_3$, and
    with high probability $[Q,\Gamma_2] = \Gamma_3$, it follows that
    with high probability $[Q,Q] = \Gamma_3$.
\end{proof}

\begin{prop}\label{prop:generic-Q}
    Let $|K/k|=e$ and $abce>1$.  Let $Q$ range over kinder of the above
    nursery $\Gamma_*$ of $\mathcal{H}_{abc}(K)$ and
    $\iota:Q\hookrightarrow \mathcal{H}_{abc}(K)$ an arbitrary embedding.
    If $ae>2$ and $\ell\approx abe/2$, with high probability
    there is an $\alpha\in \Aut(Q)$ such that
    such that $\iota(\alpha(\Gamma_2)) = \Gamma_2$.
\end{prop}

\begin{proof}
    By the Proposition~\ref{prop:generic-Q-1} we know that $\Gamma_3$ is the
    commutator subgroup of $Q$, which allows us to restrict $[,]_{\mathcal{H}}$
    to a biadditive map $[,]_Q:(Q/\Gamma_3)^{\otimes 2}\to \Gamma_3$. Note that
    $\Gamma_2$ is part of a hyperbolic pair $(E,F:=\Gamma_2)$ for $Q$.  So as
    described above, such decompositions are in bijective correspondence with
    so-called \emph{hyperbolic} idempotents $(e,1-e)\in \End(\Lambda_*(Q))$
    where $\Lambda_*(Q)$ is the coordinate representation of the $k$-bilinear map
    $[,]_Q$; see \cite{BMW:genus2}*{Lemma~3.5}. Following
    Corollary~\ref{coro:prob-end}, $\End(\Lambda_*(Q))/J\cong k\oplus k$ where
    $J$ is the Jacobson radical, or else $\End(\Lambda_*(Q))\cong \mathbb{M}_2(K)$.
    In the first case $\End(\Lambda_*(Q))/J$ has precisely
    two proper nontrivial idempotents $(e,1-e)$ and $(1-e,e)$.  By the lifting
    of idempotents, all proper nontrivial idempotents of $\End(\Lambda_*(Q))$
    are part of a hyperbolic pair.  Furthermore, all hyperbolic pairs are
    conjugate by some $1+z$, $(z,-z)\in J$.  Hence $1+z$ lifts to an
    automorphism of $Q$, which can be checked directly or compared with
    the argument in \cite{BMW:genus2}*{Theorem~3.15b}.  In the second case
    there are many proper nontrivial primitive idempotents but all are conjugate
    and one of them has the form $(e,1-e)$, so they all do.
\end{proof}

Once we know that $\Gamma_2$ and $\Gamma_3$ are isomorphism invariants for $Q$
we can consider the biadditive maps $[,]=[,]_Q : Q/\Gamma_2 \otimes
\Gamma_2/\Gamma_3 \to \Gamma_3$ and look at the algebra of operators
which act trivially on the $Q/\Gamma_2$ factor
    \begin{align*}
        \mathcal{R}_Q & = \left\{
            (g,h)\in \End(\Gamma_2/\Gamma_3)\times \End(\Gamma_3)
            \,\,\biggr\vert \,\,
            \forall q\in Q/\Gamma_2, v \in \Gamma_2/\Gamma_3,
                [q,gv] = h[q,v]
        \right\}.
    \end{align*}
Note that this definition is a permuted variant of $\End(\Phi_*)$ called the
right nucleus; see \cite{Wilson:Skolem-Noether}*{Section~1.1}. Since
$\Gamma_2/\Gamma_3$ and $\Gamma_3$ can be identified with spaces of matrices and
$Q/\Gamma_2$ can be embedded in a space of matrices, the  algebra
$\mathcal{R}_Q$ contains a copy of $\mathbb{M}_c(K)$ acting naturally on
$\Gamma_2/\Gamma_3$ and $\Gamma_3$.  As $\mathbb{M}_c(K)$ is Morita equivalent
to $K$ we can condense the system $\Phi_*$ of $(\ell\times bc)$-matrices to a
random system $\Phi_* E_{11}$ (here $E_{11}\in \mathbb{M}_c(K)$ is the matrix
with 1 in position 11 and 0 elsewhere).  The result is a random system of
$(\ell\times b)$ matrices where we may apply the counts of
Theorem~\ref{thm:prob-end} to conclude that generically
$\mathcal{R}_Q=K\otimes \mathbb{M}_c(K)$.

\begin{prop}\label{prop:generic-R}
    For generic $Q$ the algebra $R_Q\cong \mathbb{M}_c(K)$, if $\ell >2+ b/a$.
\end{prop}

\begin{proof}[Proof of Theorem~\ref{thm:generic}]
    Let $Q$ and $\tilde{Q}$ be two kinder of $\mathcal{H}_{abc}$ and $\phi:Q\to \tilde{Q}$ an
    isomorphism. By Proposition \ref{prop:generic-Q} we may assume
    $\phi(\Gamma_2)=\Gamma_2$ and $\Gamma_3=[Q,Q]=[\tilde{Q},\tilde{Q}]$.

    By the definition of $\mathcal{R}_Q$ and Proposition~\ref{prop:generic-R}, the
    biadditive map $[,]_Q$ induces a linear map
    \begin{align*}
    \iota_Q : Q/\Gamma_2 \to \Hom_{R_Q} (\Gamma_2/\Gamma_3,\Gamma_3)
        \cong \hom_{\mathbb{M}_c(K)}(\mathbb{M}_{b\times c}(K),\mathbb{M}_{a\times c}(K))
        \cong \mathbb{M}_{b\times a}(K).
    \end{align*}
    The same applies to $\tilde{Q}$.

    In light of these identifications, $\phi$ induces an isomorphism
    $\Hom_{R_Q} (\Gamma_2/\Gamma_3, \Gamma_3)\cong \Hom_{R_{\tilde{Q}}}
    (\Gamma_2/\Gamma_3,\Gamma_3)$. Since both of these spaces can be identified
    with $\mathbb{M}_{a\times b}(K)$ then $\phi$ induces a $k$-linear
    automorphism of $\mathbb{M}_{a\times b}(K)$ which can be extended to an
    automorphism of $\mathcal{H}_{abc}$.
\end{proof}

\begin{remark}
    We can summarize the above steps in a generalized manner by
    considering a coordinate-free interpretation.  First the $(a\times b)$-matrices
    $(\Phi_1,\ldots,\Phi_c)$ are replaced with multilinear maps,
    or \emph{tensors}, $t=\sum_{i,j,k}[\Phi_k]_{ij} e_i\otimes e_j\otimes e_k$
    in $K^a\otimes K^b\otimes K^c$.
    The rings $\mathcal{M}$ and $\mathcal{R}$ are universal in that
    they are the largest faithful ring representations such that
    $t\in K^a\otimes_{\mathcal{M}} K^b\otimes_{\mathcal{R}} K^c$.
    E.g. the $(a,b,c)$-matrix multiplication tensor $t$ resides naturally
    in $\mathbb{M}_{a\times b}(K)\otimes_{\mathbb{M}_b(K)}\mathbb{M}_{b\times c}(K)
    \otimes_{\mathbb{M}_c(K)}\mathbb{M}_{c\times a}(K)$ and for
    generic $\tilde{Q}\leq \mathbb{M}_{a\times b}(K)$, the
    restriction $t|_Q$ only resides in
    $\tilde{Q}\otimes_{K} \mathbb{M}_{b\times c}(K)
    \otimes_{\mathbb{M}_c(K)}\mathbb{M}_{c\times a}(K)$.
    Theorem~\ref{thm:prob-end} considers the ring $K$, and
    Proposition~\ref{prop:generic-R} recovers $\mathcal{R}=\mathbb{M}_c(K)$.
    Corollary~\ref{coro:prob-end} is necessary since, instead of $t|_Q$,
    we first recover from commutation an element of
    $\wedge^2_{\mathcal{M}} (\tilde{Q}\oplus \mathbb{M}_{b\times c}(K))
    \otimes_{\mathbb{M}_c(K)} \mathbb{M}_{c\times a}(K)$ and
    from that the structure of the ring $\mathcal{M}$ permits us to
    reconstruct a generic tensor of $\tilde{Q}\otimes_{K} \mathbb{M}_{b\times c}(K)
    \otimes_{\mathbb{M}_c(K)}\mathbb{M}_{c\times a}(K)$.

    In general operators in $\End(K^a)\times\End(K^b)\times \End(K^c)$
    acting on $K^a\otimes K^b\otimes K^c$ are called \emph{transverse}
    tensor operators.  The above argument can be restated for a larger class
    of modules, e.g. any modules for which the rings above are Azumaya algebras,
    using this generalized point of view.
\end{remark}

\subsection{Proof of Theorem~\ref{thm:sigma-iota-general}}

We now consider how large numbers of nilpotent groups $G$ obtain the theoretical
upper bound on the size of $\sigma_{\iota}(G)$ and $\sigma(G)$.

By Corollary~\ref{coro:U-count} with $a\in b+O(1)$ and $c>1$ constant, there are
groups $S_p$ of order $p^{ab+bc+ac}=p^{b^2+O(b)}$ having $\sigma_{\iota}(S_p)\in
p^{b^4/4+\Theta(b^2)}$.  For the upper bound apply Wall \cite{Wall} to show that
$\sigma(p^{\nu})\in p^{\nu^2/4+O(\nu)}$.

For a larger family consider the
subgroups $\Gamma_2<K\leq \mathcal{H}_{abc}(\mathbb{F}_p)$ with $c>1$ constant and
$a=b=\lfloor\nu/3c\rfloor$.  Let $\log_p |K|=\nu=\ell+2bc$.  There are $p^{\ell (b^2-\ell)-O(b^2)}=
p^{\nu^3/27c^2+\Omega(\nu^2)}$ isomorphism classes of such subgroups.  Furthermore,
each subgroup has $p^{\ell^2/4 + \Theta(\ell)}\subset p^{\nu^2/36+\Theta(\nu)}$
pairwise non-isomorphic subgroups containing $\Gamma_2$.

To pass to nilpotent groups $G$, fix the direct decomposition $G=\prod_{p|n}S_p(G)$
into Sylow $p$-subgroups $S_p$.  Then every subgroup $H$ of order $k$ has
$H=\prod_{p|k}S_p(H)$ and $S_p(H)\leq S_p(G)$.  So the claims hold.\qed

\section{Classical Groups}

\begin{thm} \label{thm:classical}
    Fix a prime $p$.  For $G=A_{d-1}(p^e)=\PSL_{d}(\mathbb{F}_p^e)$
    \begin{align*}
        & (d\geq 3) &
        p^{(1/64-o(1))d^4e^2}
            & \leq 
            \sigma_{\iota}(G)\leq \sigma(G)\leq p^{(1/64+o(1))d^4e^2}.
    \end{align*}
    If $G$ is one of ${^2 A}_{2m-1}(p^e)=\PSU_{2m}(\mathbb{F}_{p^e})$, 
    $C_m(p^e)=\PSp_{2m}(\mathbb{F}_{p^e})$, or $D_m(p^e)=P\Omega^+_{2m}(\mathbb{F}_{p^{e}})$
    \begin{align*}
        & (5\leq m) & 
        p^{(1/64-\omega(1))m^4e^2} & \leq \sigma_{\iota}(G)
            \leq \sigma(G)\leq p^{(1/4-o(1))m^4 e^2};\\
        & (3\leq m\leq 4) & 
        p^{(1/4-\omega(1))e^2} & \leq \sigma_{\iota}(G)
            \leq \sigma(G)\leq p^{(1/4-o(1))m^4 e^2}.
    \end{align*}
    If $G$ is one of ${^2 A}_{2m}(p^e)=\PSU_{2m+1}(\mathbb{F}_{p^e})$,  
    $B_m(p^e)=\PSO_{2m+1}(\mathbb{F}_{p^e})$, 
    ${^2 D}_m(p^e)=P\Omega^-_{2m}(\mathbb{F}_{p^{e}})$
    \begin{align*}
        & (6\leq m) & 
        p^{(1/64-\omega(1))m^4e^2} & \leq \sigma_{\iota}(G) 
            \leq \sigma(G)\leq p^{(1/4-o(1))m^4 e^2};\\
        & (4\leq m\leq 5) & 
        p^{(1/4-\omega(1))e^2} & \leq \sigma_{\iota}(G)
        \leq\sigma(G)\leq p^{(1/4-o(1))m^4 e^2}.
    \end{align*}
\end{thm}

Our upper bounds follow the technique of Pyber in \cite{Pyber:enum}*{Section~3}
and are applied only in the case of type $A$.  We expect that tight bounds for other
classical groups require both an improved lower bound by inspecting their Sylow
$p$-subgroups in place of $U_d(K)$, as well as improving the upper bound by
inspecting the solvable subgroups of general classical groups. 

\begin{lemma}\label{lem:upper}
    For fixed $p$,  $\sigma(\GL_d(p^e))\leq 
    2^{O(d^2)}p^{d^2e}p^{O(d^3e\log d)} \sigma(U_d(K))$.
\end{lemma}
  
\begin{proof}
    By work of Aschbacher-Guralnick \cite{BNV:enum}*{Theorem~16.4},
    every group $H$ is generated by a solvable subgroup $S$ and one more element
    $g\in S$.  So first we enumerate the number of solvable subgroups of
    $\GL_d(p^e)$ by first selecting a conjugacy class of a maximal solvable
    subgroup $M$.  P\'alfy shows there are at most $A=2^{O(d)}$ such
    classes \cite{BNV:enum}*{Theorem~14.1} and each conjugacy class has 
    order at most $B=|\GL_d(p^e)|\leq p^{(1-o(1))d^2e}$.  Let $U$ be the 
    $p$-core of $M$. We claim $|M/U|\leq p^{2d\log d}$.  Consequently, 
    $M=\langle U,g_1,\ldots,g_{\lceil 2d\log d\log p\rceil}\rangle$. 
    The maximum size of a solvable group is $C=p^{(1/2-o(1))d^2 e}$ attained by 
    the minimal Borel subgroups.   
    
    Altogether, each solvable subgroup $S$ resides in one of the $A$ many
    conjugacy classes of maximal solvable group $M$, so that $M$ is one of at
    most $B$ conjugates one of which contains $S$ and $U\leq U_d(K)$.  There are
    $\sigma(U_d(K))$ choices of $U$ and $C^{2d\log d \log p}$ choices of $g_i$.
    So the number of choices of $S$ is 
    \begin{align*}
        ABC^{2d\log d \log p}\sigma(U_d(K)).
    \end{align*}

    To see the bound on $|M/U|$ we proceed by Suprenenko's structure theory of
    solvable matrix groups (compare \cite{BNV:enum}*{pp.~127--128}).  Let
    $V_1>\cdots>V_{c+1} =0$ be a composition series for $V=K^d$ as an
    $M$-module.  Then $\bar{V}:=\bigoplus_i V_i/V_{i+1}$ is a semisimple
    $M$-module $\bar{M}=M/\ker_M \bar{V}$ has trivial unipotent normal
    subgroups, i.e. $\ker_M\bar{V}$ is the $p$-core $U$.  Now 
    $\bar{M}$ embeds into  $T\rtimes \Sigma$ where $T\leq \prod_{i=1}^c
    E_i\rtimes \Gal(E_i/k)$ with $E_i/K$ field extensions with $\sum_{i=1}^c
    |E_i:K|=d$, and $\Sigma\hookrightarrow \mathrm{Sym}_c$.  So $|\bar{M}|\leq
    p^{\sum_i e_i \log_2 e_i} c!$. So the claim holds.
\end{proof}

\begin{proof}
    For type $A_{d-1}(p^e)$ use the natural embedding of $U_d(K)\subset \PSL_d(K)$
    and Corollaries~\ref{coro:U_d} \& \ref{coro:U-count}.  For the remaining classical groups use Witt's
    Extension Lemma to embed $U_m(K)$ into $\Isom(\phi)$ where $\phi$ is a
    sesquilinear or quadratic form on $K^d$ where $2m\leq d\leq 2m+2$ and $m$ is the 
    Witt index of $\phi$ over the finite field $K$.  From this
    observe that this embedding has determinant $1$ and factors through
    $G=PS\Isom(\phi)$.  In all but the orthogonal groups, specifically $\PSO_d(K)$, this
    leads to a finite simple group of classical type. In the orthogonal groups
    we further take commutator subgroup of $G$.  When $d\geq 12$ and $m\geq 5$, by
    Corollary~\ref{coro:U_d}, $\sigma_{\iota}(G)\geq p^{(1/1024-o(1))d^4e^2}$.
    The remaining bounds concern $m\in \{3,4\}$, adjusted in the case of
    $\Omega^-$.

    For the upper bounds we apply Lemma~\ref{lem:upper}.
\end{proof}

\begin{remark}
    It is difficult to estimate the value of $\mu(n)$ when $n=|\PSL_d(p^e)|$
    because of the accumulation of small prime divisors of $(p^{ke}-1)$ can lead
    to primes $r$ where $\nu_r(n)>\nu_p(n)=\binom{d-1}{2}$; cf.
    Remark~\ref{rem:Mersenne}.  However, for many $d$, $\mu(n)=\nu_p(n)$ and in
    those cases the bound $\sigma(n)\leq n^{\mu(n)+1}$ implies that our above
    bounds have $\sigma(\PSL_d(p^e))$ attaining the asymptotic bound
    $\sigma(G)$ as $G$ ranges over all groups of order $n=|\PSL_d(p^e)|)$ for
    such $n$.  Jeff Achter has suggested to us that this happens for infinitely many
    $d$, perhaps even for dense set of dimensions $d$.
\end{remark}

\section{Alternating groups}

Next we estimate the isomorphism types of subgroups of the alternating groups.
Pyber \cite{Pyber:enum}*{Corollary~2.3} has shown that alternating groups
$\mathrm{Alt}_k$ of order $n=k!/2$ have
\[
    2^{k^2/16+\Omega(k)}\leq \sigma(\mathrm{Alt}_k)\leq 24^{(1/6+o(1))k^2}.
\]
By a formula of Legendre, if $p$ is prime and $p^e|k!$ then $e=\sum_{i>0} \lfloor n/p^i\rfloor$
and so
$e < \frac{k}{p-1}$.  Therefore, $\mu(k!)< k$ and so
\[
    \sigma(k!/2) \in 2^{\Theta(k^2)}.
\]
Therefore on an asymptotic log scale, alternating groups attain the maximum
possible $\sigma$.  We now prove the same for $\sigma_{\iota}$.
\begin{thm}\label{thm:alt}
    The group $\mathrm{Alt}_k$ has at least $2^{k^2/36 + \Omega(n\log n)}$ isomorphism types
    of subgroups.  In particular, for alternating and symmetric groups $G$ of order $n$,
    $\log_2 \sigma_{\iota}(G)\in \Theta(\log_2 \sigma(n))$.
\end{thm}
\begin{proof}
Similar to the count above for classical groups we proceed by counting subgroups
within a fixed group.  An obvious choice might be to consider Sylow
$2$-subgroups, however, we obtain a suitable lower bound by instead counting
with the groups $\Gamma:=\Gamma_k$ of a direct sum of $k$ copies of the
symmetric group on $3$ letters.  This acts on $3k$ points as a union of $k$
orbits.  So $\Gamma_k$ embeds in the symmetric group $\mathrm{Sym}_{3k}$ and in
turn into $\mathrm{Alt}_{3k+2}$.%
\footnote{We can pass to a subgroup of index $2$ in $\Gamma_1$ which embeds in $\mathrm{Alt}_{3k}$.}
We will show that for large $k$,
\[
    \sigma_{\iota}(\Gamma_k)\in 2^{k^2/4 + \Omega(k\log k)}.
\]

Let us consider the subgroups $H$ containing the subgroup $\Gamma_2:=[\Gamma,\Gamma] = C_3^k$.
Notice that the subgroup $\Gamma_2$ is characteristic in $H$ and the quotient
$\Gamma/\Gamma_2$ is naturally identified with $\mathbb{F}_2^k$ and comes with
Hamming distance $\varpi: \Gamma/\Gamma_2 \to [0,\dots, k]$ counting the number
of components where the corresponding coordinate is nonzero.

Now we cannot use this function directly since its definition  depends on the
embedding of $H$ into $\Gamma$. Instead, observe that the elements of
$H/\Gamma_2$ come with their action of $\Gamma_2$ and for each $h \in H$ we have
\[
    |\langle [h, g] \mid g\in \Gamma_2\rangle | = 3^{\varpi(\bar h)}
\]
This shows that the restriction of the Hamming distance to $H/\Gamma_2$ can be
determined only by the isomorphism type of the group $H$.

Finally, up to code equivalence, the number of binary codes of degree $k$ and
dimension $e$ is at least $2^{e(d-e)}/k!$.  This is maximized at $e=k/2$ where
we get $2^{k^2/4+\Omega(k\log k)}$.
\end{proof}

\begin{remark}
    Another way to rephrase this argument is that we view $\mathbb{F}_2^k$ as
    the maximal slit torus in $\GL_k(\mathbb{F}_3)$, the subgroup $H$
    corresponds to a subgroup of the torus (together its action on
    $\mathbb{F}_3^k$. If this subgroup is sufficiently large then it has no
    repeated eigenvalues and it can be diagonalized in only one way (up to a
    permutation matrix) which gives a subspace of $\mathbb{F}_2^k$ modulo the
    action of $S_k$. Since the number of subspaces is of the order of
    $2^{O(k^2)}$ which is significantly larger than $|S_k|$, this leads to a
    lower bound for the number of isomorphism types of subgroup which is of the
    order of $2^{O(k^2)}$.
\end{remark}

\begin{remark}
    The is a variant of this construction using $2$-groups 
    -- instead of working with $S_3$ one can use $D_8 \subset S_4$. In
    this case the analog $\Gamma_2$, the group $C_4^k$ has slightly more complicated
    definition (since it is not a Sylow subgroup)%
    \footnote{There is a unique element $c$ in the center of $H$, such that the set $\{g \in H | g^2 =c\}$ has exactly $2^k$ elements.
    The group $\Gamma_2$ is generated by all elements $g\in H$ such that $g^2=c$.}.
    This  leads to at least
    $2^{n^2/64 + \Omega(n\log n)}$ isomorphism types of $2$-subgroups inside $\mathrm{Alt}_n$.

    This can be generalized further to $p$-groups for odd primes $p$ using $C_p \wr
    C_p \subset S_{p^2}$, in this case the analog of $\Gamma_2$ is defined as the
    unique elementary abelian subgroup of $H$ of size $p^{kp}$.
\end{remark}

\section{Small Rank and Exceptional Groups}

We now consider the finite simple groups not covered in Theorems~\ref{thm:classical}
\& \ref{thm:alt}.

\subsection{Small rank counts}
We turn now to the language Steinberg groups.  Indeed the method for Steinberg
groups can be used for classical groups as well albeit with less sharp bounds.
We recall terminology from Steinberg, see \cite{CFSG}.

We consider the subgroup $U_{\Delta}(R)$ of the Steinberg group
generated by the root subgroups corresponding to the positive roots, and we let
$G_{\Delta}(R)$ be the corresponding Steinberg group.  Note that $U_{\Delta}(R)$
is nilpotent and any element in it can be written uniquely as a product
$\prod_{\alpha \in \Delta^+} e_\alpha(r_\alpha)$. There are variations of this
construction for twisted root systems, if the ring $R$ has a suitable
automorphism, also in certain case these groups can be defined when the ring $R$
is non-commutative.

The cases not covered in the previous section are all groups with bounded Lie
rank.  So our aim is to prove that, on a log scale,
$\sigma_\iota(G_{\Delta}(R))$ is comparable $\sigma(R,+)$ when the rank of
$\Delta$ is bounded.  Since $|G_{\Delta}(R)|\in |R|^{O(1)}$ the results will
follow.

If the Dynkin diagram for the root system $\Delta$ contains
$A_2$ as a subdiagram then $U_\Delta(R)$ contains $U(3,R)$ as a subgroup, thus
we can use the results for classical groups to deduce that the Heisenberg group
$U_\Delta(\mathbb{F}_{p^e})$ contains at least $p^{\Omega(e^2/4)}$ isomorphism types
of subgroups.  This bound is easy to apply (and as in Remark~\ref{rem:best-bounds-U}
it applies only once $e\geq 9$).  Tighter bounds would appear to require detailed
study of the subgroups of unipotent groups of exceptional and twisted groups.
We encourage such work but do not pursue it here.

Excluding the groups in Theorem~\ref{thm:classical} and groups of Lie type with
a diagramatic embeddings of $A_2$ excludes all but the following cases: $A_1$,
${}^2A_2$, ${}^2A_3$, ${}^2A_4$, $B_2$, ${}^2B_2$, ${}^3D_4$, ${}^2F_4$, $G_2$,
${}^2G_2$.  The groups ${}^2 B_2$, ${^2 F}_4$, and ${^2 G}_2$ exist only in for
fields of orders $2^{1+2n}$, $2^{1+2n}$, and $3^{1+3n}$ respectively.  We can
avoid going over all cases, by relaxing the condition of containing $A_2$ as a
subdiagram, to the containment of the corresponding root systems (maybe over a
slightly smaller field).  This leaves only
\begin{align*}
    A_1,\quad {}^2A_2,\quad B_2,\quad {}^2B_2,\quad {}^2F_4,\quad {}^2G_2.
\end{align*}
We now cover these cases.

\subsection{Type $A_1$ ($\PSL_2(\mathbb{F}_{p^e})$)}

The Sylow $p$-subgroups of $\PSL_2(\mathbb{F}_{p^e})$ are isomorphic to
$\mathbb{Z}_p^e$ and so $\sigma(\PSL_2(\mathbb{F}_{p^e})\in p^{\Omega(e^2)}$.
Since $|\PSL_2(\mathbb{F}_{p^e})|\in p^{\Theta(e^2)}$ this bound is sufficient
for our estimates. However the $p$-subgroups are elementary abelian and so
$\sigma_{\iota}(\mathbb{Z}_p^e)=e+1$. In fact the subgroups of
$\PSL_2(\mathbb{F}_{p^e})$ have been classified. The maximal subgroups are
either upper triangular, dihedral, or $\mathrm{Alt}_5$.  The isomorphism types of subgroups
of $\mathrm{Alt}_5$ is bounded, and the subgroups of dihedral groups are dihedral or
cyclic.  Both of these groups are characterized up to isomorphism by their
orders so they contribute at most $\log |\PSL_2(\mathbb{F}_{p^e})|$ distinct
isomorphism classes. Finally the group of upper triangular matrices is isomorphic
to $K^{\times}\ltimes K$.  If $H\leq K^{\times}\ltimes K$ then $H\cap U_2(K)$ is
normal in $H$ and $|H:H\cap U|=|HU:U|$ divides $|K^{\times}|$ is prime to $p$.
So $H\cap U$ is the Sylow $p$-subgroup of $H$.  As $H$ is solvable, it
has a Hall $p'$-subgroup $Q$ and so $H=Q\ltimes_{\rho} (H\cap U)$.  Furthermore $Q$ is
drawn from the subgroups of the cyclic group $K^{\times}$ and $H\cap U$ is drawn from
subspaces of $(\mathbb{Z}/p)^e$.  Therefore the isomorphism of subgroups of a fixed
order determined by the conjugacy classes of the images of the maps
$\rho:\mathbb{Z}/m\to K^{\times}\hookrightarrow \GL_e(\mathbb{F}_p)$.  Such images are conjugate
if they have the same order.  So in total $\sigma_{\iota}(\PSL_2(\mathbb{F}_{p^e}))\in O(e^2\log p)$.

\begin{remark}\label{rem:Mersenne}
    If we constrain the rank and exponent and allow only the prime to vary then in
    general the diversity of subgroups of simple groups is severely limited.  For
    instance, for every Mersenne prime $p=2^k-1$, $n=|\PSL_2(\mathbb{F}_p)|$ has $\mu(n)\geq
    k+1$ and so $\sigma_{\iota}(n)\in 2^{\Theta((\log n)^2)}$.  Yet
    $\sigma_{\iota}(\PSL_2(\mathbb{F}_p))\in O(\log n)$.
\end{remark}

\subsection{Type ${}^2A_2$ ($\mathrm{PSU}_3(\mathbb{F}_{p^{2e}})$)}
We note that for $p>3$, the Sylow $p$-subgroups of $\mathrm{PSU}_3(\mathbb{F}_{p^{2e}})$ are
isomorphic to the Sylow $p$-subgroups of $\mathrm{PSL}_3(\mathbb{F}_{p^e})$.
However, an estimate for all $p$ is to use $F=\mathbb{F}_{p^{2e}}$ with
quadratic field involution $\sigma$. To apply Theorem~\ref{thm:nursery-count}
use $R = F^\sigma = \{ \alpha \in F \mid \alpha=\alpha^*\}$, $M=\{\alpha \in F
\mid \alpha+\alpha^* = 0\}$, ${\sf S}=\{1,\omega\}$ where
$F=\mathbb{F}_{p^e}[\omega]$ and ${\sf T}$ is any non-zero element in $M$.
Then apply Theorem~\ref{thm:nursery-count}.

\subsection{Type $B_2$, $\mathbf{char}\,F\neq 2$ (${\rm
PSO}_5(\mathbb{F}_{p^e})$)} Again this follows form
Theorem~\ref{thm:nursery-count} this time with $R=M= F$, notice that the natural
commutator map is not the usual multiplication, but $(r,m)\to 2 rm$.  However if
the characteristic is not equal to $2$, then the multiplication is isotopic to
the usual one. An other way ro rephrase this is to say that $U_{B_2}(F)$ contain
a subgroup isomorphic to $U_3(F)$.

\subsection{Type $B_2$, $\mathbf{char}\, F =2$ $({\rm PSO}_5(2^e))$}
The main difference is between
$B_2(F)$ when ${\rm char}~F\neq 2$ and ${\rm char}~F=2$ is that the nilpotency
class of group $U:=U_{\Delta}$ drops form $3$ to $2$ and the center becomes
larger. This prevents us from applying Theorem~\ref{thm:nursery-count} which
needs three steps to form the nursery. However, we can  modify the proof of that
theorem to split the center as a direct sum.

In this case $U$ is an extension of $F^2$ by $F^2$ where the
commutator bi-map is:
$$
[(r,s),(\tilde{r},\tilde{s})] = (r\tilde{s} - \tilde{r}s, r\tilde{s}^2 - \tilde{r}s^2).
$$
We will also work with the quadratic map $\phi:F^2 \to F^2$ given by $\phi((r,s)) =
(rs,rs^2)$.

The new series we consider for our nursery is
\begin{align*}
    \Gamma_1 & = U\\
    \Gamma_2 & = \{((r,0), z) \mid r\in F, z\in Z(U)\}\\
    \Gamma_3 & = Z(U)=[U,U]U^2\cong \{(r,s)\mid r,s\in F\}\\
    \Gamma_4 & = \{(0,s)\mid s\in F\}\\
    \Gamma_5 & = 1.
\end{align*}
To be precise, we use the epimorphisms $\alpha:U/\Gamma_3\to F^2$ and the
isomorphism $\beta:Z(U)\to F^2$; so, $\Gamma_2=\alpha^{-1}(F,0)$ and
$\Gamma_4=\beta^{-1}(0,F)$.  Once more we shall be interested in kinder
$Q$ where $\Gamma_2\leq Q\leq \Gamma_1$ and we shall need to expose what data in
addition to the structure of $Q$ recovers $\Gamma_i$.

Let $F=\mathbb{F}_2[\omega]$ and consider kinder $Q$ that contain,
$\alpha^{-1}(\{(\omega^i,0)\mid i\in \{-1,0,1\}\}$).
This condition implies that $Z(Q)= [Q,Q] = \Gamma_3$.
The quadratic map $\phi:Q/\Gamma_3 \to \Gamma_3$ has two maximal totally singular subspaces
(since $\phi((r,s)) = 0$ if, and only if, $r=0$ or $s=0$) and the preimage of larger
one in $Q$ is the subgroup $\Gamma_2$.  Thus $\Gamma_2\leq Q$ can be characterized
as the largest elementary abelian subgroup of $Q$.  Hence, $\Gamma_2$ and $\Gamma_3$
are isomorphism invariants of kinder $Q$ containing $\alpha^{-1}\{(\omega^i,0)\mid i\in \{-1,0,1\}\}$.

The first step is to construct enough elements in $\Gamma_1/\Gamma_2$ that will
allow us to identify this space with the field $F$.  For $i\in \{-1,0,1\}$ choose
representatives $A_i\in \alpha^{-1}(\omega^i,0)$ and $B_i\in \alpha^{-1}(0,\omega^i)$.
We proceed by inductively defining further coset representatives $A_i$,
by appealing to the following recurrence relation:
$$
[A_{k+1},B_{-1}][A_k,B_0] = [A_{k-2},B_1][A_{k-1},B_0].
$$
This definition depends only on the initial choice of $A_i$, and $B_i$ for $i\in
\{-1,0,1\}$, and commutation in $Q$ -- which is an isomorphism invariant.

The elements $A_i$ allow us to identify $\Gamma_1/\Gamma_2$ with the field $F$,
but after that to decompose the center $\Gamma_3$ as a direct sum $F\oplus F$.
The subgroup $\Gamma_4$ can be characterized as the subgroup generated by
$[A_{k+1},B_{-1}][A_k,B_0]$ when $k$ varies,  similarly we can identify its
complement as the subgroup generated by $[A_{k+1},B_{-1}][A_{k-1},B_0]$ when $k$
varies. Finally we can identify both $\Gamma_3/\Gamma_4$ and $\Gamma_4$ with $F$
by sending the $[A_k,B_0]\Gamma_0$ to $\omega^k$ and $[A_{k+1},B_{-1}][A_k,B_0]$
to $\omega^k + \omega^{k-1}$. After all these we can map $Q$ to a additive
subgroup of $F$ by sending $Q$ to image of $[Q,A_0]$ inside $\Gamma_3/\Gamma_4$.

Thus, the number of isomorphism types of subgroups of $Q$ is at least the number
of additive subgroups of $F$ which contain $1,\omega,\omega^{-1}$ divided by the
number of possible choices for the elements $A_i$ and $B_i$ (only modulo
$\Gamma_3$). Since $A_i$ and $B_i$ are constrained to be in the two totally
singular subspaces, the number of choices for the $A_i$ is $p^{3e}$ and for the
$B_i$ is $|Q/\Gamma_3|^3$.

Therefore the number of isomorphism types of subgroups of $Q$ such that
$|Q/\Gamma_2| = 2^\ell$ is at least at least $2^N$ where
$$
N = (e-\ell)( \ell-3) - 3e - 3\ell
$$
This is maximized when $\ell \approx e/2$ and gives $N = O(e^2/4)$
This bound is trivial for $e\leq 24$ and becomes nontrivial once $e > 24$.

\subsection{Type ${}^2B_2(\mathbb{F}_{2^{2e+1}})$, Suzuki groups}
Theorem~\ref{thm:simple} purposefully omits Suzuki groups because we have
no proof of a bound for this case.  What we include here is a reduction of the
enumeration to a purely field property which we have verfied computationally
for all fields $\mathbb{F}_{2^{1+2e}}$ with $1+2e\leq 1001$.  This suggests
to us that Suzuki groups indeed also fit a bound of the type reported in
Theorem~\ref{thm:simple}.
\begin{quote}
    Assume that in a field $\mathbb{F}_{2^{1+2e}}$ there is a subset
    $S$ such that $|S|\leq 3\sqrt{1+2e}$ and
    $\{xy^{2^{e+1}}-yx^{2^{e+1}}\mid x,y\in S\}$
    contains a basis for $\mathbb{F}_{2^{1+2e}}$.
\end{quote}

The group $U=U_{\delta}(F)$ is an extension of the additive group of $F$ by $F$
where the square is given by the $\mathbb{F}_2$-quadratic map
$$
    \phi(x) =x^{1+2^{e+1}}
$$
(after the standard identification of the center and the abelianization with $F$).

Fix a minimal $\mathbb{F}_2$ subspace $V$ of $F$, such that the image of $V$
under $\phi$ spans the whole $F$ as an $\mathbb{F}_2$ vector space, which by our
assumption permits $\dim V\leq 3\sqrt{1+2e}$.  Fix a basis $v_1,\dots,v_f$ for
$V$. (Evidently $\dim V\geq \sqrt{1+2e}$ at minimum.  The size of $V$ is the
so-called \emph{Sims rank} of $U$; cf. \cite{BNV:enum}*{Section~5.2}.)

We will count the subgroups $H$ of $U$ whose projection into the $U/[U,U]$ contains
the space $V$, together with specified elements $h_1,\dots h_f$ which project to
$v_i$ in $U/[U,U]$. Taking the squares of all possible products of $h_i$ we can
identify the subgroups $U^2$ (which is equal to $[U,U]$) with the field $F$ (of
course this is only possible if all these elements satisfy the necessary linear
relations), the important observation is that this identification does not
depend on the embedding of $H$ in $U$). Using this identification we can
identify the quotient $H/H^2$ with the subspace of $F$  which contains $V$. The
number of such spaces is of the order of $2^{(2e+1-f)^2/4}$ and at most
$2^{(2e+1+f)f/2}$ of these correspond to the same isomorphism type  (the number
of choices of $h_i$ modulo $H^2$). Thus the number of isomorphism types of
subgroups of $U$ is at least $2^N$ where
$$
N = e^2 - 2ef = e^2 - O(e^{3/2})
$$

As in the other cases this bound is trivial for the first few Suzuki groups, and
as we stress it is proved under the above assumption on fields of order $2^{1+2e}$.

\subsection{Type ${}^2F_4(\mathbb{F}_{2^{2e+1}})$, large Ree groups}
These groups contain a Heisenberg subgroup over $F$ and so a
suitable lower bound on the number of isomorphism types comes from the above
count of Heisenberg groups; cf.~\cite{CFSG, Ti}.

\subsection{Type ${}^2G_2(F_{3^{2e+1}})$, small Ree groups}
Finally, the small Ree groups contain a subgroup with the same associate graded Lie ring as the  Heisenberg group
(this can be seen from the commutation relations in~\cite{HSvM})  over $F$ and  we may appeal to Theorem~\ref{thm:nursery-count}
with $R=M= F$.  If $F = F_{3^{2e+1}}$ this leads to about $3^{(2e-1)^2/4 - 8e-4}$
different isomorphism types of subgroups in $U$ while the number of subgroups is
bounded above by $3^{(9-o(1))e^2}$.


\end{document}